\documentclass[11pt,a4paper]{article}
\usepackage{amsfonts}
\usepackage{enumerate}
\usepackage{inputenc, graphicx, t1enc, latexsym, amsmath, amsfonts, oubraces, amsthm, amssymb}
\usepackage[english]{babel}
\usepackage{multirow}
\usepackage{layout}
\newtheorem*{ther}{First Conservation law of elastic strips}
\newtheorem*{theo}{Second Conservation law of elastic strips}
\newtheorem*{1.theorem}{Theorem 1}
\newtheorem*{2.theorem}{Theorem 2}
\newtheorem*{3.theorem}{Theorem }
\newtheorem{definition}{Definition}
\newtheorem{lemma}{Lemma}
\newtheorem{remark}{Remark}

\newtheorem*{corollary}{Corollary}
\newtheorem{proposition}{Proposition}

\begin{document}

\title{ Elastic strips}
\author{ David Chubelaschwili and Ulrich Pinkall}
\maketitle

\begin{abstract}
    Motivated by the problem of finding an explicit description
    of a developable narrow Möbius strip of minimal bending energy, which was first formulated by M. Sadowsky in 1930,     we will develop the theory of elastic strips.
    Recently E.L. Starostin and G.H.M. van der Heijden found a numerical description for an elastic Möbius strip, but      did not give an integrable solution. 
    We derive two conservation laws, which describe the equilibrium equations of elastic strips.
    In applying these laws we find two new classes of integrable elastic strips which correspond to spherical elastic
    curves. We establish a connection between Hopf tori and force--free strips, which are defined by one of the            integrable strips, we have found. We introduce the P--functional and relate it to elastic strips.
\end{abstract}


\section{Introduction}
Sadowsky [6] showed that the bending energy~$E(F_{\gamma}) = \int_{M}H^2 dA $ of an infinitely narrow developable strip is proportional to 
\begin{align}
S(\gamma)=\int_{0}^{L}\kappa^2(1+\lambda^2 )^2 ds .
\end{align}
  We define elastic strips as critical points of (1), among all variations leaving the length fixed. E.L. Starostin and G.H.M. van der Heijden [7] generalized the variational problem of minimizing the energy of developable strips with finite width. They derived first integrals by using the variational bicomplex and 
obtained six balance equations for the components of the internal force $F$ and moment $M$ in the direction of the Frenet frame
\begin{align}
F'+\omega\times F=0,~~~~M'+\omega\times M+T\times F=0.
\end{align}
 By using a computer software they found a numerical description of an elastic Möbius strip, but they did not give explicit formulas and integrable solutions.
From the integrable geometric point of view it turned out to be more convenient to compute the internal force $b_0$ and torque $b_1$ in a fixed coordinate system, so that $b_0$ and $b_1$ become conservation fields along elastic strips themselves. We derive these conservation fields, which are all based on an low technology approach. The conservation laws enable us to find two new integrable systems of elastic strips. These elastic strips correspond to spherical elastic curves. Furthermore, we introduce the P-functional and show that the tangent image of the centerline of an elastic strip is a critical point of the P-functional, which enables us to reduce the variational problem to spherical geometry. In summery, in this paper we prove:

\begin{ther}
    A strip is elastic iff the force vector 
    \begin{align}
    b_0= a_1 T + a_2 N + a_3 B
    \end{align}
     is constant, with
        \begin{align}
           \begin{split}\label{koef}
                a_1  &:= \frac{1}{2} (\kappa^2(1+\lambda^2)^2+\mu)) \\ 
                a_2  &:= \kappa'(1 +\lambda^2 )^2+2\kappa(1 +\lambda^2 )\lambda\lambda'\\
                a_3  &:= -\left(\kappa^2(1+\lambda^2 )^2 \lambda + (\frac{\kappa'}{\kappa}(1+\lambda^2)2\lambda)'
                     +((1+\lambda^2 )2\lambda)''\right).
            \end{split}
        \end{align}
\end{ther}

\begin{theo}   For an elastic strip the torque vector
   \begin{align}
       b_1 = s_1 T+ s_2 N+s_3 B-\gamma\times b_0  
    \end{align}
    is constant, whereby
    \begin{align}\begin{split}\label{s_i}
     	s_1 &:= 2\kappa\lambda(1+\lambda^2) \\
     	s_2 &:= \frac{1}{\kappa}(2\kappa\lambda(1+\lambda^2))' \\
     	s_3 &:= \kappa(1+\lambda^2)(1-\lambda^2).\\
    \end{split}\end{align}
\end{theo}

By applying the conservation laws we find two classes of integrable systems, namely elastic momentum strips, which are defined by $ s_1'=0$, and force--free strips. We prove

\begin{1.theorem}
    For an elastic momentum strip the binormal $B$ of $\gamma$ is a spherical elastic curve. Conversely for each such arclength
    parametrized spherical curve $B\colon [0,\hat{L}]\rightarrow S^2$ with non-vanishing geodesic curvature $\lambda$ and
    $T:=B\times B'$ the space curve
    \begin{align}\label{21}
       \gamma(t)=\int_{0}^{t}(1+\frac{1}{\lambda^2(s)})T(s)ds
    \end{align}
    defines an elastic momentum strip.

\end{1.theorem}

\begin{2.theorem} 
    For a force--free strip the tangent vector $T$ of $\gamma$ is a spherical elastic curve with Lagrange multiplier 1.
    Conversely for each such spherical arclength parametrized curve $T\colon [0,\tilde{L}]\rightarrow S^2$ with geodesic curvature
    $\lambda$ the space curve
    \begin{align}\label{12}
        \gamma(t)=\int_{0}^{t}(1+\lambda^2(s))T(s)ds
    \end{align}
    defines a force--free strip. 
\end{2.theorem}
 We present two different methods to prove the last theorem. The first method uses the conservation laws, while the second method does not even require the calculus of variation, but only an elegant argument. The second argument can be generalized and reduces the problem of finding the centerline $\gamma$ of an elastic strip to it's tangent image.

\section{Elastic strips}

	Let $ \gamma\colon[0,L]\rightarrow \mathbb{R}^{3} $ be regular Frenet curve with velocity $v = \left|\gamma'\right|$.
	Denote
	
	\begin{align}
	    \kappa & = \frac{\left|\gamma'\times\gamma''\right|}{\left|\gamma'\right|^3} \\
	    \tau &= \frac{det(\gamma',\gamma'',\gamma''')}{\left| \gamma' \times \gamma'' \right|^2}\\
	    \lambda & =\frac{\tau}{\kappa} 
	\end{align}
	the curvature, torsion and modified torsion of $ \gamma $.
	\begin{center}
\begin{align}\begin{split}\label{frenet}
{\ \quad \ } \gamma' &=  v T\\
{\ \quad \ } T' &=							\qquad  v \kappa N						\\
{\ \quad \ }N' &= - v \kappa T	 				\qquad	 + v \lambda \kappa B\\
{\ \quad \ } B' &= 						\qquad -v \lambda \kappa N. 
\end{split} \end{align}

\end{center}

We investigate ruled surfaces described by
\begin{align}
 F_{\gamma}:[-\epsilon,\epsilon] \times [0,L]\mapsto\mathbb{R}^3,~~ F_{\gamma}(t,u)= \gamma(t) + u D(t),
 \end{align}
   where $ D(t) = \lambda(t)T(t) + B(t)$~denotes the modified Darboux vector.
 One can show that this surface is developable and~$\gamma $~is a pregeodesic of the surface. We call these surfaces
 rectifying strips. We investigate rectifying infinitely narrow strips which are critical points of the Willmore-functional $E(F_{\gamma}) = \int_{M}H^2 dA $ among  all space curves with fixed end points and $\dot{L}:=\left.\frac{\partial}{\partial t}\right|_{t=0}L(\gamma_t) =0$. Wunderlich [8] showed that the limit $\epsilon\rightarrow 0$ $\int_{M}H^2 dA$ is proportional
 to the Sadowsky functional~$
   S(\gamma)=\int_{0}^{L}\kappa^2(1+\lambda^2 )^2 ds $. 
   This gives rise to the following 
   \begin{definition}
   A strip $F_{\gamma}$ is elastic, if $\gamma$ is a critical point of the modified Sadowsky functional
 \begin{align}\label{sadow}
  S_\mu(\gamma)=\int_{0}^{L}(\kappa^2(1+\lambda^2 )^2-\mu) v dt,
 \end{align}
 where $\mu$ is a Lagrange multiplier, standing for the length constraint.
A Frenet curve $ \gamma\colon(0,L)\rightarrow \mathbb{R}^{3} $ defines an elastic strip, if $\gamma$ defines an elastic strip on each closed subinterval of $\left(0,L\right)$.
 \end{definition}
\begin{remark}
 A helpful observation is that $\tilde{\gamma}(s):=\tilde{\mu}\gamma(s)$ is a critical point of $S_\frac{\mu}{\tilde{\mu}^2}$ iff $\gamma$ is a critical point of $S_\mu.$ In fact, since $\tilde{\kappa}=\frac{\kappa}{\tilde\mu}$,
 $\tilde{\lambda}=\lambda,$ $\tilde{v}=\tilde{\mu}v$ and the scaling of $\gamma$ is compatible with our boundary condition, we obtain $\frac{1}{\tilde{\mu}}S_\mu=S_\frac{\mu}{\tilde{\mu}^2}$. By scaling the curve one can always achieve $\mu=-1$ or $\mu=1$.
\end{remark} 

\section{Conservation laws of elastic strips}

First we start with a technical variational
\begin{lemma}
    Let ~$ \gamma_0 : [0,L]\mapsto \mathbb{R}^3$ be an arclength parametrized Frenet curve
   and $ \gamma: [-\epsilon,\epsilon] \times [0,L] \mapsto \mathbb{R}^3$ a variation of~$\gamma_0$~with variational field
    $\dot{\gamma}(s):=\left.\frac{\partial}{\partial t}\right|_{t=0}\gamma_{t}(s)= u_1(s)T(s)+u_2(s)N(s)+u_3(s)B(s)$. Then
    \begin{align}
        \begin{split}\label{variationsformel}
    	    \dot{v} &= u_1'-\kappa u_2\\
            \dot{\kappa}&= u_1 \kappa' + u_2 (\kappa^2(1-\lambda^2)) - 2 u_3' \lambda \kappa - u_3(\lambda \kappa)'+
             u_2''\\ 
            \dot{\lambda} &= u_1\lambda'+u_2(\frac{(\lambda\kappa)''}{\kappa^2}-
            \frac{(\lambda\kappa)'\kappa'}{\kappa^3}+\lambda^3\kappa+\lambda\kappa) 
            +u_2' (2\frac{\lambda'}{\kappa}+\frac{(\lambda\kappa)'}{\kappa^2})\\
            &\quad \, +u_2''\frac{\lambda}{\kappa}-u_3\lambda\lambda'+u_3'(1+\lambda^2)-u_3''\frac{\kappa'}{\kappa^3}+
            u_3'''\frac{1}{\kappa^2}.
        \end{split}
    \end{align}
\end{lemma}

\begin{proof} Since $(\dot{\gamma})' = (\gamma')^\cdot$ we have
    \begin{align*}
        \dot{v} T + \dot{T} = (u_1'-u_2\kappa)T+(u_1\kappa +u_2'-\lambda \kappa u_3)N+(u_3'+\lambda\kappa u_2)B.
    \end{align*}
    Hence
    \begin{align*}
        \dot{v} & = u_1'- u_2\kappa\\
        \dot{T} & =( u_1\kappa +u_2'- \lambda\kappa u_3)N+(\lambda\kappa u_2+ u_3')B. 
    \end{align*} 
Furthermore $(T')^\cdot = (\dot{T})'$ yields
\begin{align*}
    \dot{v}\kappa N+\dot{\kappa} N+\kappa\dot{N} =
    &-(u_1\kappa+u_2'-\lambda\kappa u_3)\kappa T  \\
    & + \left((u_2'+ u_1\kappa-\lambda\kappa u_3)'-(\lambda\kappa u_2 +  u_3')\lambda\kappa\right)N\\
     &+\left((\lambda\kappa u_2+ u_3')'+\lambda\kappa(u_1\kappa + u_2'-\lambda\kappa u_3)\right)B.
\end{align*}
By comparing the coefficients we obtain
  \begin{align*}\begin{split}
      \dot{\kappa}&= u_2''+ u_1'\kappa + u_1\kappa'-(\lambda\kappa)' u_3-\lambda\kappa u_3'-u_3'\lambda\kappa
      -(\lambda\kappa)^2u_2-u_1'\kappa
       +u_2\kappa^2\\
       &=u_1\kappa'+(1-\lambda^2)\kappa^2 u_2 +u_2''-(\lambda\kappa)'u_3-2\lambda\kappa u_3'\\
       \dot{N}&=-(u_2'+u_1\kappa-u_3\lambda\kappa)T+\frac{1}{\kappa}
        \left((u_2'+ u_1\kappa-\lambda\kappa u_3)\lambda\kappa + (u_3'+\lambda\kappa\ u_2)'\right)B.
   \end{split}\end{align*}
Finally $(\dot{N})'=(N')^\cdot $ gives 
  \begin{align*}
    \begin{split}
      (N')^\cdot &=(-v\kappa T+\lambda \kappa v \,B)^\cdot \\
       &=(-v\kappa)^\cdot T-\kappa\dot{T}+(\dot{\lambda}\kappa+\lambda\dot{\kappa}+\lambda\kappa\dot{v})B
          +\lambda\kappa\dot{B}~\\
         \Rightarrow\; \left\langle (N')^\cdot ,B\right\rangle&=-\kappa(\lambda\kappa u_2+ u_3')+\dot{\lambda}\kappa
          +\lambda\dot{\kappa}+\lambda\kappa\dot{v}~,\\
        \left\langle (\dot{N})',B\right\rangle&=(\frac{1}{\kappa}\left(( u_1\kappa+ u_2'-\lambda\kappa u_3)
          \lambda\kappa+(\lambda\kappa u_2+ u_3')\right))'\\
         \Rightarrow\;\dot{\lambda}\kappa&=(\frac{1}{\kappa}\left((u_1\kappa+ u_2'-\lambda\kappa u_3)\lambda\kappa+(\lambda\kappa u_2+ u_3')'\right))' \\&\quad -\lambda\dot{\kappa}-\lambda\kappa\dot{v}+\lambda\kappa^2 u_2+\kappa u_3'\\
 &=u_1(\kappa\lambda)'+u_1'\kappa\lambda+u_2(\frac{1}{\kappa}(\lambda\kappa)''-\frac{(\lambda\kappa)'\kappa'}{\kappa^2}-(        1-\lambda^2)\lambda\kappa ^2+2\lambda\kappa^2 )\\
        & \quad +u_2'(2\lambda'+ \frac{(\lambda\kappa)'}{\kappa}) +u_2''\lambda-(\lambda^2\kappa)'u_3-\lambda^2\kappa u_3'
+\frac{1}{\kappa}u_3'''-\frac{\kappa '}{\kappa^2}u_3''\\& \quad-\lambda u_1 \kappa' +\lambda(\lambda\kappa)'u_3+2\lambda^2\kappa u_3'-\lambda\kappa u_1'+\kappa u_3'.\\
\end{split}\end{align*}
      \begin{align*}\begin{split}
          \dot{\lambda}&=u_1\frac{1}{\kappa}((\kappa\lambda)'-\lambda\kappa')+u_1'(\lambda\kappa-\lambda\kappa)
          \\& \quad+u_2\left(\frac{(\lambda\kappa)''}{\kappa^2}-
           \frac{(\lambda\kappa)'\kappa'}{\kappa^3}-(1-\lambda^2)\lambda\kappa+2\lambda\kappa\right) 
            +u_2'(2\frac{\lambda'\kappa}{\kappa^2}+\frac{(\lambda\kappa)'}{\kappa^2})\\& \quad + u_2''\frac{\lambda}{\kappa}            +u_3(-\frac{(\lambda^2\kappa)'}{\kappa} +\frac{\lambda}{\kappa}(\lambda\kappa)')+
           u_3'(1+\lambda^2)+u_3''(-\frac{\kappa'}{\kappa^3})+\frac{1}{\kappa^2}u_3'''\\
 &=u_1\lambda'+u_2\left(\frac{(\lambda\kappa)''}{\kappa^2}-\frac{(\lambda\kappa)'\kappa'}{\kappa^3}+\lambda^3\kappa+\lambda\kappa\right)\\
          & \quad+u_2'\left(2\frac{\lambda'}{\kappa}+\frac{(\lambda\kappa)'}{\kappa^2}\right)
          +u_2''\frac{\lambda}{\kappa}
          -u_3\lambda\lambda'
            +u_3'(1+\lambda^2)-u_3''\frac{\kappa'}{\kappa^3}+u_3'''\frac{1}{\kappa^2}.
      \end{split}\end{align*}
\end{proof}
  In the following, we  do not distinguish between the symbols ~$\gamma_0$~and the variation~$\gamma$ while computing the first variation formula for the integrand of the modified Sadowsky functional.
\begin{lemma}\mbox{}
    Let $\gamma$$\colon$$[0,L]\mapsto \mathbb{R}^3$~be an arclength parametrized curve that defines an elastic strip.     Consider a variation~$\gamma$~with variational field\\ $\dot{\gamma}= u_1T+u_2N+u_3B$, then
      \begin{equation}\label{2}
       \frac{1}{2} \left.\frac{\partial}{\partial t}\right|_{t=0} (\kappa_t^2 (1+\lambda_t^2)^2-\mu) v_t=u_2 f_1 + u_3        f_2+ b'
      \end{equation}
      with
        \begin{align}
        \begin{split}\label{3}
         f_1&:= (\kappa'(1 +\lambda^2 )^2+2\kappa(1 +\lambda^2 )\lambda\lambda')'+ \frac{\kappa}{2} (\kappa^2(1+\lambda^2 )^2+\mu) \\ 
      &\quad +\lambda \kappa \left( \kappa^2(1 +\lambda^2 )^2  \lambda+(\frac{\kappa'}{\kappa}(1 +\lambda^2 )2\lambda)'+(1 +\lambda^2 ) 2  \lambda)'' \right)\\
   \end{split} \end{align}
 
    \begin{align}
   \begin{split}\label{4}
   f_2 &:=-\left(\kappa^2(1+\lambda^2 )^2 \lambda + (\frac{\kappa'}{\kappa}(1+\lambda^2)2\lambda)'+((1+\lambda^2 )2\lambda)''\right)'\\
     	 &~~~+\kappa\lambda\left(\kappa'(1+\lambda^2 )^2+2\kappa(1 +\lambda^2 )\lambda\lambda'\right)\\
   \end{split} \end{align}

 \begin{align}
 \begin{split}\label{5}
  b	&:= u_1(\frac{1}{2}(\kappa^2(1+\lambda^2)^2-\mu)\\
   							& \quad+ u_2 \left( (6 \lambda \lambda' \kappa +2\lambda^2 \kappa')(1+\lambda^2) - \left(\kappa(3\lambda^2+1)(1+\lambda^2)\right)' \right)\\ 
   							& \quad+ u_2' (\kappa (3 \lambda^2+1)(1+\lambda^2)) \\
   							& \quad+ u_3\left( (2 \lambda \frac{\kappa'}{\kappa}(1+\lambda^2))' + (2 \lambda (1+\lambda^2))'' \right) \\
   							& \quad- u_3' (2 \lambda\frac{\kappa'}{\kappa}(1+\lambda^2) + (2\lambda(1+\lambda^2))')+u_3'' \left( 2(1 + \lambda^2) \lambda \right). 
 \end{split} \end{align}   
\end{lemma} 
 \begin{proof}
\begin{align*}
   	&  \frac{1}{2} \left.\frac{\partial}{\partial t}\right|_{t=0} (\kappa_t^2 (1+\lambda_t^2)^2-\mu) v_t\\ 
   	&= \frac{1}{2} (\kappa^2(1+\lambda^2)^2-\mu)  \dot{v} + {\dot\kappa} \kappa (1+\lambda^2)^2 
   		 	 + 2 \kappa ^2(1+\lambda^2) \lambda {\dot\lambda} \\ 
	  & \stackrel{(\ref{variationsformel})} {=} u_1 \left( \kappa' \kappa \left(1 +\lambda^2\right)^2 + 2  \kappa^2 \lambda \lambda' (1 + \lambda^2)  \right)
		 + u_1'\left(\frac{1}{2} \left( \kappa^2(1 + \lambda^2)^2 - \mu \right) \right)\\
		& \quad + u_2 \left( \frac {1}{2} \kappa^3 (1 + \lambda^2)^2 (1 + 2\lambda^2) + \frac{1}{2}\mu\kappa + 2(1+ \lambda^2)\lambda
				\left( (\lambda\kappa)'' - (\lambda\kappa)'\frac{\kappa'}{\kappa} \right)\right)\\
		& \quad	+ u_2'\left(2(1+\lambda^2)\lambda\left( 2\lambda'\kappa + (\lambda\kappa)' \right) \right)
		 + u_2''\left(\kappa (1 + \lambda^2)^2 + 2 \kappa(1 + \lambda^2)\lambda^2 \right)\\
		& \quad + u_3 \left( -\kappa(1+\lambda^2)^2 (\lambda\kappa)' - 2 \kappa^2(1 + \lambda^2) \lambda^2 \lambda' \right)\\
		& \quad + u_3'\left( -2\kappa^2(1 + \lambda^2)^2 \lambda + 2 \kappa^2(1+ \lambda^2)^2 \lambda \right)\\
		& \quad + u_3'' \left( -2 (1 + \lambda^2) \lambda \frac{\kappa'}{\kappa} \right)
		  + u_3''' \left( 2(1 + \lambda^2) \lambda \right)\\
		&= + u_2\left((\kappa'(1 +\lambda^2 )^2+2\kappa(1 +\lambda^2 )\lambda\lambda')' +\frac{\kappa}{2}(\kappa^2(1 +\lambda^2 )^2+\mu)\right.\\  
                  & \quad \left.+\lambda \kappa \left( \kappa^2(1 +\lambda^2 )^2  \lambda 
								 + (\frac{\kappa'}{\kappa}(1 +\lambda^2 )2\lambda)'+ (1 +\lambda^2 ) 2  \lambda)'' \right)\right)\\
								 & \quad + u_3\left(-\left(\kappa^2(1+\lambda^2 )^2 \lambda + (\frac{\kappa'}{\kappa}(1+\lambda^2)2\lambda)'+((1+\lambda^2 )2\lambda)'' \right)'\right.\\
     	        & \left. \quad +\kappa \lambda\left( \kappa'(1 +\lambda^2 )^2+2\kappa(1 +\lambda^2 )\lambda\lambda'\right) \right)\\
     	        & \quad + \left(u_1(\frac{1}{2}(\kappa^2(1+\lambda^2)^2-\mu)\right.\\
   							& \quad \left.+ u_2 \left( (6 \lambda \lambda' \kappa +2\lambda^2 \kappa')(1+\lambda^2) - \left(\kappa(3\lambda^2+1)(1+\lambda^2)\right)' \right)\right.\\ 
   							& \left.\quad+ u_2' (\kappa (3 \lambda^2+1)(1+\lambda^2))\right. \\
   							& \left.\quad+ u_3\left( (2 \lambda \frac{\kappa'}{\kappa}(1+\lambda^2))' + (2 \lambda (1+\lambda^2))'' \right)\right. \\
   							& \left.\quad- u_3' (2 \lambda\frac{\kappa'}{\kappa}(1+\lambda^2) + (2\lambda(1+\lambda^2))')+ u_3'' \left( 2(1 + \lambda^2) \lambda \right)\right)'. 
          \end{align*}
Comparing the expression above with (\ref{3}), (\ref{4}) and (\ref{5}) we obtain (\ref{2}) as desired.
 \end{proof}

   \begin{proposition}\mbox{}
 \begin{enumerate}
 \item
 The critical points of $S_\mu$ are characterized by the Euler-Lagrange equations $f_1=f_2=0$.\\
 \item
 If $\gamma$ is a critical point of~$S_\mu$, then for each variation of~$\gamma$,~which leaves the integrand of the Sadowsky functional $(\kappa_t^2(1+\lambda_t^2)^2-\mu)v_t$ invariant, one obtains $b'=0$.
\end{enumerate}

   	\end{proposition}
   	
\begin{proof}\mbox{}
  \begin{enumerate}
\item Let $\gamma$ be a critical point of~$S_\mu$. Since the Sadowsky functional is invariant under reparameterizations, we can assume $\gamma$ to be arclength parametrized. From (\ref{2}) we obtain for each proper variation of $\gamma$  
 \begin{align*}
  0=\left.\frac{\partial}{\partial t}\right|_{t=0}S_\mu(\gamma_t)&=\int_{0}^{L}(u_2(s) f_1(s) + u_3(s) f_2(s)+ b'(s))ds\\
                                                                 &=\int_{0}^{L}(u_2(s) f_1(s) + u_3(s) f_2(s)+)ds+ b(L)-b(0).\\
 \end{align*}
   Using the fact that~${b}(L)={b}(0)=0$~ for a proper variation, we obtain the desired Euler-Lagrange equations~$f_1=f_2=0$. 
 \item
   		
   			The invariance of $(\kappa_t^2(1+\lambda_t^2)^2-\mu) v_t$ with respect to $t$ implies:
   			
   	\begin{eqnarray*}
   		0	&=& \frac{1}{2} \left.\frac{\partial}{\partial t}\right|_{t=0} (\kappa_t^2(s) (1+\lambda_t^2(s))^2-\mu) v_t(s)\\ 
   	 &  \stackrel{(\ref{2})}{=}& u_2(s) f_1(s) + u_3(s) f_2(s)+ b'(s) \\
   		 	&= & b'(s).
   	\end{eqnarray*}
   			   		
 \end{enumerate}
  			
\end{proof}
  Recently Th. Hangan [\ref{h1}] derived the  two Euler-Lagrange equations of (1), while his second equation coincides with     ours, it seems unlikely that his first equation is equivalent to $f_1$.\\
  It is apparent from the Euler-Lagrange equations that planar critical points of the modified Sadowsky functional      are just planar elastic curves. Furthermore helices solve the Euler-Lagrange equations.
 To obtain more solutions we use the symmetries of the Sadowsky functional in the spirit of the Noether theorem. 
   Obviously the Euclidian group leaves $(\kappa^2(1+\lambda^2)^2-\mu) v$ invariant.
   This transformation group is generated by translations and rotations. For a variation consisting only of
   translations we obtain:
   
   \begin{align*}\begin{split}
   		\gamma_{t}(s)											&= \gamma(s)+ t a \qquad \mbox{for an arbitrary } a\in \mathbb{R}^3\\
   		\Rightarrow\qquad \dot{\gamma}	&= a=\underbrace{\left\langle a,T\right\rangle}_{u_1}T + \underbrace{\left\langle a,N\right\rangle}_{u_2} N + \underbrace{\left\langle a,B\right\rangle}_{u_3}B.
   \end{split}\end{align*}
  
   Computing $u_2',u_3',u_3''$ one gets 
   \begin{align*}\begin{split}
   	u_2' &= -\kappa\left\langle a,T\right\rangle + \lambda   \kappa \left\langle a,B\right\rangle\\
   	u_3' &= -\lambda \kappa \left\langle a,N\right\rangle\\
   	u_3''&= -(\lambda \kappa)' \left\langle a,N\right\rangle + \lambda \kappa^2\left\langle a,T\right\rangle-\lambda^2\kappa^2\left\langle a,B\right\rangle.
   \end{split}\end{align*}
   
   From (\ref{5}) we obtain that
    \begin{align}\begin{split}
    	b &= \left\langle  a, \frac{1}{2}(\kappa^2(1+\lambda^2)^2+\mu)T + (\kappa'(1+\lambda^2)^2+ 2 \kappa \lambda' \lambda(1+\lambda^2))  N \right.\\
    					& \quad \left. - \left( \kappa^2 (1+\lambda^2)^2 \lambda + \left(2\frac{\kappa'}{\kappa}\lambda(1+\lambda^2)\right)'+ (2 \lambda (1+ \lambda^2) )''\right)  B\right\rangle
    \end{split}\end{align}
   
   is constant for any $a$ $\in$ $\mathbb{R}^3$. This however implies that
   
   	\begin{align*}\begin{split}
   	b_0 &:=\frac{1}{2}(\kappa^2(1+\lambda^2)^2+\mu)T + (\kappa'(1+\lambda^2)^2+ 2 \kappa \lambda' \lambda(1+\lambda^2))  N  \\
    		& \quad -\left( \kappa^2 (1+\lambda^2)^2 \lambda + \left(2\frac{\kappa'}{\kappa}\lambda(1+\lambda^2)\right)'+ (2 \lambda (1+ \lambda^2) )''\right) B
   \end{split}\end{align*}
   is constant.\\
   \\~
      By taking into account that 
   $\left.\frac{\partial}{\partial t}\right|_0 A_t \gamma(s) = w \times \gamma(s) $ for $w \in  \mathbb{R}^3$ and $A_t \in SO(3)$ one obtains in a quite similar way that
   \[	b_1 = s_1 T+ s_2N+s_3B-\gamma\times b_0 \]is constant for elastic strips, where 
   \begin{align}\begin{split}
     	s_1 &:= 2\kappa\lambda(1+\lambda^2) \\
     	s_2 &:= \frac{1}{\kappa}(2\kappa\lambda(1+\lambda^2))' \\
     	s_3 &:= \kappa(1+\lambda^2)(1-\lambda^2).\\
     	\end{split}\end{align} 
         In the following we show that elastic strips are characterized by $b_0$ being constant.~More precisely, we show that the constants of $b_0$
         is equivalent to the Euler--Lagrange equations.

   \begin{ther}\mbox{}
{\ }\\
    A strip is elastic iff the force vector $b_0= a_1 T + a_2 N + a_3 B$~    is constant, with
    \begin{align}\begin{split}
     a_1	&:= \frac{1}{2} (\kappa^2(1+\lambda^2)^2+\mu)) \\ 
     a_2	&:= \kappa'(1 +\lambda^2 )^2+2\kappa(1 +\lambda^2 )\lambda\lambda'\\
     a_3	&:= -\left(\kappa^2(1+\lambda^2 )^2 \lambda + (\frac{\kappa'}{\kappa}(1+\lambda^2)2\lambda)' +((1+\lambda^2 )2\lambda)'' \right).
    \end{split}\end{align}
\end{ther}

   \begin{proof}\mbox{}
   
   	It suffices to show 
   	\begin{align}\label{b_0'}
   	b_0'=f_1 N + f_2B ,
   	\end{align} 
   	since $ b_0 $~is constant iff $b_0'=0$ iff $f_1=f_2=0$ iff $\gamma $ defines an elastic strip. Using the Frenet formulas we get
      
    \begin{align}\begin{split}
      b_0'	&= \underbrace{(a_1'-\kappa a_2)}_{= 0} T
      					+\underbrace{(a_2' + \kappa a_1 -\kappa\lambda a_3)}_{f_1} N 
      					+\underbrace{(a_3'+ \kappa\lambda a_2)}_{f_2} B \\
      			&= f_1 N + f_2 B.
    \end{split}\end{align}
    
    Since $a_2 = \frac{1}{\kappa}  a_1'$ the T-coefficient drops out and one calculates
    
    \begin{align*}\begin{split}
    	a_2'+\kappa a_1-\kappa\lambda a_3	&= (\kappa'(1 +\lambda^2 )^2+2\kappa(1 +\lambda^2 )\lambda\lambda')'\\
    																		& \quad+ \frac{\kappa}{2}(\kappa^2(\lambda^2+1)^2+\mu ) \\
    																		& \quad+ \kappa \lambda   \left( \kappa^2(1+\lambda^2)^2  \lambda 
    																			 + \left(\frac{\kappa'}{\kappa} (1+\lambda^2) 2\lambda \right)'
    																		+ ( (1+\lambda^2) 2 \lambda )''\right)\\ 
    																		&=  f_1
    \end{split}\end{align*}

    \begin{align*}\begin{split}
      a_3'+\lambda\kappa a_2	&=-\left(\kappa^2(1+\lambda^2 )^2 \lambda + \left(\frac{\kappa'}{\kappa}(1+\lambda^2 )2\lambda\right)' +((1+\lambda^2 )2\lambda)'' \right)'\\
     													&  \quad+\kappa \lambda\left( \kappa'\left((1 +\lambda^2 )\right)^2+2\kappa(1 +\lambda^2 )\lambda\lambda'\right)\\
    													&= f_2~.
    \end{split}\end{align*}
    
    \end{proof}

    \begin{theo}
    \mbox{}
{\ }\\
     	For an elastic strip the torque vector $b_1 = s_1 T+ s_2 N+s_3 B-\gamma\times b_0$ is constant, whereby\\
     	 \begin{align*}\begin{split}
     	s_1 &:= 2\kappa\lambda(1+\lambda^2) \\
     	s_2 &:= \frac{1}{\kappa}(2\kappa\lambda(1+\lambda^2))' \\
     	s_3 &:= \kappa(1+\lambda^2)(1-\lambda^2).\\
     	\end{split}\end{align*} 
     	
     	Furthermore, if $b_1$ is constant but $ \gamma$ does not define 
      an elastic strip, then $ \left|\gamma\right|$ is conserved.
    \end{theo}

\begin{proof}\mbox{}

First we show
\begin{align}\label{not elasic}
b_1'=-\gamma \times b_0'. 
\end{align}
This implies that $b_1$ is conserved.
Using the Frenet formulas one obtains

\begin{align*}
\begin{split}
	b_1' 	&=  (s_1'-\kappa s_2) T \\
				& + (\underbrace{\kappa s_1+s_2'-\lambda  \kappa  s_3}_{-a_3} -\underbrace{\left\langle T\times b_0,N\right\rangle}_{-a_3}) N\\
				& + (\underbrace{s_3'+\lambda\kappa s_2}_{a_2}- \underbrace{\left\langle T\times b_0,B\right\rangle}_{a_2}) B -\gamma \times b_0'.\\
\end{split} \end{align*}

Since ~$s_2=\frac{1}{\kappa} s_1' $~, the T-coefficient drops out.
 Using
\begin{align*}
	\left\langle T\times b_0,N\right\rangle &= \left\langle N\times T,b_0\right\rangle=-a_3\\
	\left\langle T\times b_0,B\right\rangle &= \left\langle B\times T,b_0\right\rangle=\quad a_2
\end{align*}
and the definition of $b_0,$ one shows that the terms in front of N and B vanish as well:\\

The N term vanishes since
\begin{align} \begin{split}\label{a}
	\kappa s_1+s_2'-\lambda\kappa s_3		&= 2\kappa^2(1+\lambda^2)\lambda+(\frac{1}{\kappa}(2\kappa\lambda(1+\lambda^2 ))')'\\
																			&\quad-\lambda\kappa^2(1+\lambda^2)(1-\lambda^2) \\
																			&= \lambda \kappa ^2 (1+\lambda^2)^2 + \left(\frac{\kappa'}{\kappa}2 \lambda(1+\lambda^2)\right)'\\
																			&  \quad+ (2 \lambda(1+\lambda^2))''= -a_3~.\\
\end{split}\end{align}

The B term vanishes since
\begin{align} \begin{split}\label{s_2}
 s_3'+\lambda\kappa s_2 
 &= \left( \kappa(1+\lambda^2)(1-\lambda^2) \right)' + \lambda \left(2 \lambda \kappa (1+\lambda^2)\right)' \\
 &= \kappa'(1+\lambda^2)(1-\lambda^2)+\kappa((1+\lambda^2)(1-\lambda^2))' \\
 				& \quad + 2\kappa'\lambda^2(1+\lambda^2)+2\kappa\lambda(\lambda(1+\lambda^2))'\\
 &= \kappa'(1+\lambda^2)^2+\kappa(-4\lambda^3\lambda'+6\lambda^3\lambda'+2\lambda\lambda') \\
 		 &= \kappa' (1+\lambda^2)^2 + 2 \kappa(1+\lambda^2)\lambda\lambda'\\
 &=  a_2~.
\end{split}\end{align}
This shows that~$b_1'=-\gamma \times b_0'.$ Therefore $b_1$ is constant if $b_0$ is constant.\\
 Conversely we assume $b_1$ is constant but $\gamma$ does not define an elastic strip. From (\ref{b_0'}) we obtain $b_0'=f_1N+f_2B,$ which implies
\[0=b_1'=-\gamma\times(f_1N+f_2B).\] \\ 
Hence~$\gamma$~lies in the span of N and B. In particular we obtain\\
\[\left\langle \gamma,\gamma\right\rangle'=2\left\langle \gamma,T\right\rangle=0~.\]
\end{proof}

 \begin{proposition}
	Let $\gamma$ define an elastic strip such that $a_1$ is constant, then $\gamma$ is a cylindrical helix.
\end{proposition}

\begin{proof}\mbox{}

	For an elastic strip $b_0$ is constant. Since~$\left\langle b_0,T\right\rangle=a_1$ is constant, 
	we conclude that~$\gamma$~is a slope line.
	This implies~$\lambda$~to be constant.
	Since 
	
	\[a_1=\frac{1}{2}(\kappa^2(1+\lambda^2)^2+\mu)~\]
	
	 is constant,~$\kappa$~ must be constant as well.
	Therefore the strip is defined by a cylindrical helix.\\
	\end{proof}

	The following proposition was already known by Th. Hangan and C. Murea [\ref{h2}]. Since they worked with different    Euler-Lagrange equations, we give a proof which only uses the conversation laws.
	
 \begin{proposition}\mbox{}
 Let $\gamma$$\colon$$[0,L]\mapsto \mathbb{R}^3$ be an arclength parametrized non-planar geodesic on a cylinder with non-constant curvature, then $\gamma$
 defines an elastic strip iff the planar curve 
 \begin{align}
 \tilde{\gamma}(s)=\gamma(as)-\frac{\lambda}{a}s(\lambda T+B),
 \end{align}
 with $a:=\sqrt{1+\lambda^2}$, is an elastic curve with zero energy, i.e. $0=\tilde{\kappa}'^2 + \frac{1}{4}\tilde{\kappa}^4+l\tilde{\kappa}^2$ for some $l<0$.
  \end{proposition}
 
 \begin{proof}
 
 Geodesics on a cylinder are slope lines in $\mathbb{R}^3$. Thus $\lambda$ and $\lambda T+B$ are constant. If $\gamma$ defines an elastic strip, then $b_0$ and $b_1$ are conserved. Therefore
 \begin{align} \begin{split}
 \left\langle b_1,\lambda T+B\right\rangle &= \lambda s_1+s_3-\left\langle \gamma\times b_0,\lambda T+B\right\rangle\\
                                           &= \kappa(1+\lambda^2)^2-\left\langle \tilde{\gamma},b_0\times(\lambda T+B)\right\rangle,\\
\end{split} \end{align}
hence 
\begin{align}\label{BÖ4}
\left\langle \tilde{\gamma},\frac{1}{1+\lambda^2}b_0\times(\lambda T+B)\right\rangle= \tilde{\kappa}-\left\langle b_1,\frac{1}{1+\lambda^2}(\lambda T+B)\right\rangle ,
\end{align}
where $\tilde{\kappa}(s)=a^2\kappa(as)$.
This shows that the distance from $\tilde{\gamma}$ and the axis $(\lambda T+B)\times (b_0\times(\lambda T+B))$ is proportional to it's curvature, which is a characterization for planar elastic curves. Thus, there exist an $E\in \mathbb{R}$ with
\begin{align} \begin{split}
E &=\tilde{\kappa}'^2(s) + \frac{1}{4}\tilde{\kappa}^4(s)+l\tilde{\kappa}^2(s)\\
  &=a^6\kappa'^2(as) + \frac{a^8}{4}\kappa^4(as)+la^4\kappa^2(as)\\
\end{split}\end{align}
or equivalently
\begin{align}\label{Bö}
\kappa'^2(s) + \frac{1+\lambda^2}{4}\kappa^4(s)+\frac{l}{1+\lambda^2}\kappa^2(s)=\frac{E}{(1+\lambda^2)^3}.
\end{align}
From (\ref{Bö}) we obtain
\begin{align} \label{BÖ1}
(\frac{\kappa'}{\kappa})'=-\frac{E}{(1+\lambda^2)^3\kappa^2}-\frac{1}{4}(1+\lambda^2)\kappa^2.
\end{align}
Computing 
\begin{align} \begin{split}\label{tüb}
\left\langle b_0,\lambda T+B\right\rangle &=\lambda a_1+a_3\\
                                          &=\frac{1}{2}\lambda \mu-2(\frac{\kappa'}{\kappa})'(1+\lambda^2)\lambda -\frac{1}{2}\kappa^2(1+\lambda^2)^2\lambda 
\end{split}
\end{align}
yields
\begin{align}\label{BÖ2}
2(1+\lambda^2)\lambda(\frac{\kappa'}{\kappa})'+\frac{1}{2}\kappa^2(1+\lambda^2)^2\lambda =\frac{1}{2}\lambda \mu -\left\langle b_0,\lambda T+B\right\rangle.
\end{align}
Plugging (\ref{BÖ1}) in (\ref{BÖ2}) we get
\begin{align}
-\frac{2\lambda E}{(1+\lambda^2)^2\kappa^2}=\frac{1}{2}\lambda \mu -\left\langle b_0,\lambda T+B\right\rangle.
\end{align}
Since $\gamma$ is a non-planar curve with non-constant curvature, we obtain $E=0$.
Conversely if $\tilde{\gamma}$ is an elastic curve with zero energy then it is apparent from (\ref{BÖ1}) and (\ref{tüb}) that $\left\langle b_0,\lambda T+B\right\rangle=\frac{1}{2}\lambda \mu$ is constant. Hence
\begin{align}
0=\left\langle b_0',\lambda T+B\right\rangle=\left\langle f_1N+f_2B,\lambda T+B\right\rangle=f_2.
\end{align}
From (\ref{BÖ4}) one sees that $\left\langle b_1,\lambda T+B\right\rangle$ is conserved as well and therefore
\begin{align} \begin{split}
0=\left\langle b_1'(as),\lambda T+B\right\rangle &=-\left\langle \gamma(as) \times (f_1(as) N(as)+f_2(as) B(as)),\lambda T+B\right\rangle\\
                                             &=-f_1(as)\left\langle \gamma(as),N(as)\times \lambda T+B\right\rangle\\
                                             &=-f_1(as)a\left\langle \tilde{\gamma}(s),\tilde{T}(s)\right\rangle\\
                                             &=-f_1(as)a\frac{1}{2}\left\langle \tilde{\gamma}(s),\tilde{\gamma}(s)\right\rangle'.
\end{split}
\end{align}
This shows that $f_1$ vanishes as well.
 \end{proof}

 \section{Momentum strips}

\begin{definition}
A curve $\gamma$ defines a momentum strip, if 
\begin{align}
\left\langle b_1+\gamma\times b_0,T\right\rangle
\end{align}
 is a constant non-zero function.
\end{definition}
In [\ref{ls}] it is shown that an arclength parametrized spherical elastic curve with geodesic curvature $\lambda$ satisfies
\begin{align}\label{f_3}
 \lambda'^2+\frac{1}{4}\lambda^4+(1-\frac{l}{2})\lambda^2=A,
 \end{align}
 where $l$ denotes the Lagrange multiplier and $A$ an arbitrary constant.

\begin{1.theorem}
For an elastic momentum strip with Lagrange multiplier $\mu$ the binormal $B$ of $\gamma$ is a spherical elastic curve with Lagrange multiplier $-\mu$. Conversely for each such arclength parametrized curve $B\colon [0,\hat{L}]\rightarrow S^2$  with non-vanishing, non-constant geodesic curvature $\lambda$ and $T:=B\times B'$, the space curve

\begin{align}
\gamma(t)=\int_{0}^{t}(1+\frac{1}{\lambda^2(s)})T(s)ds
\end{align}

defines an elastic momentum strip with
\begin{align}
 S_{\mu}(\gamma)=\int_{0}^{\hat{L}}(1+\lambda^2)dt-\mu L(\gamma).
\end{align}
\end{1.theorem}

\begin{proof}
By scaling the curve $\gamma$ one can achieve $s_1=\left\langle b_1+\gamma\times b_0,T\right\rangle=2$, thus
\begin{align}
\kappa=\frac{1}{\lambda(1+\lambda^2)}
\end{align}
\begin{align} \begin{split} \label{m_5}
\left\langle b_0,b_0\right\rangle &=\frac{1}{4}(\frac{1}{\lambda^2(s)}+\mu)^2+ \frac{1}{\lambda^4(s)}\lambda'^2(s)(1+\lambda^2(s))^2+\frac{1}{\lambda^2(s)}\\
                                  &=\frac{\tilde{\lambda}'^2(t)}{\tilde{\lambda}^4(t)}+\frac{1}{4}\frac{1}{\tilde{\lambda}^4(t)}+\frac{1}{\tilde{\lambda}^2(t)}(\frac{\mu}{2}+1)+\frac{\mu^2}{4},
 \end{split} \end{align}

 where $\tilde{\lambda}(t):=\lambda(s(t))$, $t'(s):={1+\lambda^2(s(t))}$. Using the Frenet equations (\ref{frenet}) it is apparent that $t$ is the arclength parameter of $B$ and $\frac{1}{\tilde{\lambda}}$ its curvature.
(\ref{m_5}) is equivalent to
\begin{align}\label{m_1}
-\frac{1}{4}=\tilde{\lambda}'^2(t)+\tilde{\lambda}^4(t)(\frac{\mu^2}{4}-\left\langle b_0,b_0\right\rangle´) +(\frac{\mu}{2}+1)\tilde{\lambda}^2(t).
\end{align}

Obviously any solution of (\ref{m_1}) has no zeros, thus we obtain
\begin{align} \label{m_2}
(\frac{1}{\tilde{\lambda}(t)})'^2+\frac{1}{4}(\frac{1}{\tilde{\lambda}(t)})^4+(\frac{\mu}{2}+1)(\frac{1}{\tilde{\lambda}(t)})^2=-(\frac{\mu^2}{4}-\left\langle b_0,b_0\right\rangle).
\end{align}
 In particular we obtain from (\ref{m_2}) that B is a spherical elastic curve with Lagrange multiplier $-\mu$ for an elastic momentum strip.
Conversely, let B be such an arclength parametrized spherical elastic curve with non-vanishing, non-constant geodesic curvature $\lambda$, then one can easily check from
\begin{center}
$
\begin{array}{ccccc}\label{Frenetm}
{\ \quad \ } \qquad  B'  &=&							 &- N	&					\\
{\ \quad \ } \qquad -N' &=& \lambda T	& 					&- B \\
{\ \quad \ } \qquad  T'  &=& 						&  \lambda N &
\end{array}
$
\end{center}
that $\gamma(t)=\int_{0}^{t}(1+\frac{1}{\lambda^2(s)})T(s)ds$ has curvature $\kappa=\frac{1}{\frac{1}{\lambda}(1+\frac{1}{\lambda^2})}$ and modified torsion $\frac{1}{\lambda}$. Hence $\gamma$ defines a momentum strip. It remains to show that $\gamma$ defines an elastic strip. Consider the arclength reparametrized curve 
$\tilde{\gamma}(s)=\gamma(t(s))$, with $t'(s):=\frac{1}{1+\frac{1}{\lambda^2(s)}}$. Substituting $\frac{1}{\tilde{\lambda}}$ for $\lambda$ in the first equation of (\ref{m_5}) yields
\begin{align} \label{m_6}
\left\langle b_0,b_0\right\rangle=\lambda'^2(t(s))+\frac{1}{4}\lambda^4(t(s))+(1+\frac{\mu}{2})\lambda^2(t(s))+\frac{\mu^2}{4}.
\end{align}
Since $B$ is a spherical elastic curve with Lagrange multiplier $-\mu$ we get that $\left\langle b_0,b_0\right\rangle$ is conserved. Furthermore $\left\langle b_0,b_1\right\rangle$ is constant due algebraic reasons:
\begin{align} \begin{split}
\left\langle b_0,b_1\right\rangle &= s_1a_1+s_2a_2+s_3a_3\\
                                  &=\lambda^2+\mu -\lambda^2(1-\frac{1}{\lambda^2})\\
                                  &=\mu +1.
\end{split}\end{align}
Therefore
\begin{align}\begin{split}
0&=\left\langle b_0,b_1\right\rangle'\\
&= \left\langle b_0',b_1\right\rangle+\left\langle b_0,-\gamma \times b_0'\right\rangle\\
&=\left\langle b_0',b_1\right\rangle-\left\langle b_0',b_0\times \gamma\right\rangle\\
&=\left\langle b_0',b_1+\gamma \times b_0\right\rangle\\
&=\left\langle f_1N+f_2B,b_1+\gamma\times b_0\right\rangle\\
&=f_2\left\langle B,b_1+\gamma\times b_0\right\rangle\\
&=f_2s_3\\
&=f_2(1-\frac{1}{\lambda^2})\lambda.
\end{split}\end{align}
Since $\lambda$ is a non constant solution of (\ref{m_6}) we get $f_2=0$. $\left\langle b_0,b_0\right\rangle$ being constant yields
\begin{align}
0=\left\langle b_0,b_0\right\rangle'=\left\langle f_1N,b_0\right\rangle=f_1a_2=f_1 \lambda'(1+\frac{1}{\lambda^2}).
\end{align}
Hence $f_1$ vanishes as well and $\gamma$ defines an elastic strip.
\end{proof}

\section{Force-free strips}

\begin{definition}
An elastic strip is called force--free, if~$b_0=0$~.
\end{definition}
From $a_1\equiv0$ it is evident that $\mu<0$. By scaling the curve one can achieve that $\mu=-1$.

\begin{lemma}
Let $\gamma$ be a curve with non-constant modified torsion $\lambda=\frac{\tau}{\kappa}$. Then the following conditions
are equivalent:
\begin{enumerate}
\item $\gamma$ defines
a force-free strip,\\
\item $b_1=2\lambda T+ 2\lambda'(1+\lambda^2)N+ (1-\lambda^2)B $ is constant,\\
\item   $a_1\equiv0$ and $\left\langle J,J\right\rangle$ is conserved, $J:=s_1T+s_2N+s_3B$.
\end{enumerate}

\end{lemma}

     \begin{proof}
     It remains only to prove that the third condition implies the first.
      From $a_1\equiv0$ we obtain $a_2\equiv0$ and
      \begin{align}\label{f_0}
           \kappa=\frac{1}{1+\lambda^2}.
      \end{align}
      With (\ref{f_0}) one computes
      \begin{align}\label{f_1}
      \left\langle J,J\right\rangle=(4\lambda'^2+1)(1+\lambda^2)^2.
      \end{align}
      From (\ref{a}) and (\ref{s_2}) one checks $J'=-a_3N+a_2B$, hence
      \begin{align} 
         \begin{split}
         0=\left\langle J,J'\right\rangle &=\left\langle J,-a_3N+a_2B\right\rangle\\
                                          &=-a_3s_2+a_2s_3\\
                                          &=-a_3s_2\\
                                          &=-a_32(1+\lambda^2)\lambda'.
      \end{split} \end{align}
      From (\ref{f_1}) it follows that $\lambda$ is a non-constant elliptic function, which implies $a_3=0$.
	   \end{proof}

\begin{2.theorem} 
For a force--free strip, the tangent vector $T$ of ~$\gamma$ is a spherical elastic curve with Lagrange multiplier 1. Conversely for each such spherical arclength parametrized curve $T\colon [0,\tilde{L}]\rightarrow S^2$ with geodesic curvature $\lambda$ the space curve
 \begin{align}\label{gamma}
 \gamma(t)=\int_{0}^{t}(1+\lambda^2(s))T(s)ds,
 \end{align}
 defines a force--free strip, with 
 \begin{align}\label{wert}
 S_{-1}(\gamma)=\int_{0}^{\tilde{L}}2(1+\lambda^2)dt=2L(\gamma). 
\end{align}
\end{2.theorem}
\begin{proof}\mbox{} 
 Let $\gamma$, with arclength parameter $s$, define a force--free strip. We already know from (\ref{f_1}) that $(4\lambda'^2+1)(1+\lambda^2)^2\equiv\left\langle b_1,b_1\right\rangle$. Applying the Frenet formulas (\ref{frenet}) it is apparent that~$\lambda$~is the geodesic 
 curvature of the spherical curve $T.$ Consider the reparametrized tangent vector
  \[\tilde{T}(t):=T(s(t))~with~ s'(t)=1+\lambda^2(s(t)),~\tilde{\lambda}(t)=\lambda(s(t)).\]
  Now one calculates:
  \begin{align} \begin{split}\label{15}
  & \tilde{ \lambda}'^2(t)+\frac{1}{4}\tilde{ \lambda}^4(t) +\frac{1}{2}\tilde{ \lambda}^2(t)\\
  &=\lambda'^2((s(t))(1+\lambda^2(s(t)))^2+\frac{1}{4}\lambda^4(s(t))+\frac{1}{2}\lambda^2(s(t))\\
  &=\frac{1}{4}\left(4\lambda'^2(s(t))(1+\lambda^2(s(t))^2+\lambda^4(s(t))+2\lambda^2(s(t))+1-1\right)\\
  &=\frac{1}{4}(4\lambda'^2(s(t))+1)((1+\lambda^2(s(t))^2)-\frac{1}{4}\\
  &=\frac{1}{4}\left\langle b_1,b_1\right\rangle-\frac{1}{4}.
  \end{split} \end{align}
  From (\ref{f_0}) one observes easily that~$\left|\tilde{T}'\right|=1.$ (\ref{15}) ensures that $\tilde{T}$ is a 
  spherical elastic curve with Lagrange multiplier 1.\\
   Conversely let $T$ be such an arclength parametrized spherical elastic curve with geodesic curvature $\lambda$, then the Frenet equations are
  \begin{center}
\begin{align} \begin{split}
{\ \quad \ }\qquad T' &=					\qquad  N					\\
{\ \quad \ }\qquad N' &= -T	 					        \qquad  +\lambda B \\
{\ \quad \ }\qquad B' &= 					\qquad -\lambda N. 
\end{split} \end{align}
\end{center}
One can easily check that (\ref{gamma}) has curvature $\kappa=\frac{1}{1+\lambda^2}$ and modified torsion $\lambda$, therefore $a_1=0$. Consider the arclength reparamertized curve
 \begin{align}\label{13}
 \tilde{\gamma}(s):=\gamma(t(s)), ~with~~ t'(s)=\frac{1}{1+\lambda^2(t(s))}.
 \end{align}
 (\ref{f_1}) yields
\begin{align} \begin{split}
\left\langle J,J\right\rangle &= (4\tilde{\lambda}'(s)+1)(1+\tilde{\lambda}^2(s))(1+\tilde{\lambda}^2)^2\\
                              &= (4\frac{\lambda'^2(t(s))}{(1+\lambda^2(t(s)))^2}+1)(1+\lambda^2(t(s))^2\\
                              &= 4\lambda'^2(t(s))+1+2\lambda^2(t(s))+\lambda^4(t(s)).
\end{split}
\end{align}
Therefore we obtain

\begin{align}\label{f_4}
\frac{\left\langle J,J\right\rangle-1}{4}=\lambda'^2(t(s))+\frac{1}{4}\lambda^4(t(s))+\frac{1}{2}\lambda^2(t(s)).
\end{align}
(\ref{f_4}) shows that $\left\langle J,J\right\rangle$ is conserved, since $T$ is a spherical elastic curve with Lagrange multiplier 1. The claim follows now from the previous lemma.
 \end{proof}

  \begin{figure}[!ht] \small
 \begin{center}
 \includegraphics[scale=0.35]{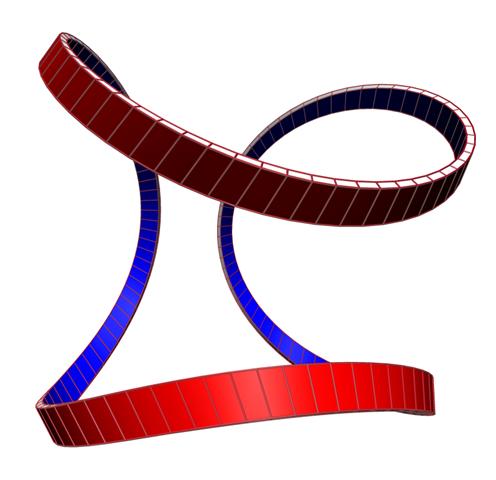} \includegraphics[scale=0.35]{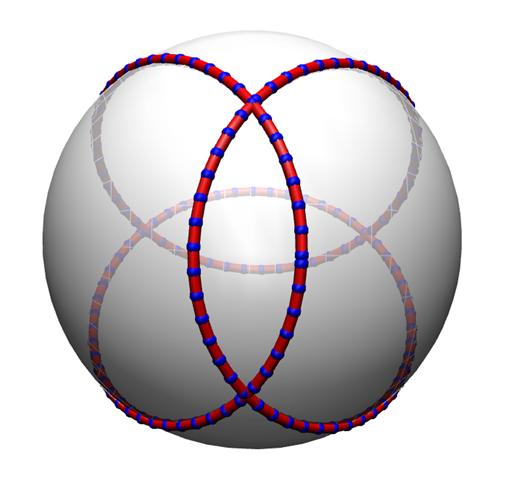}
 \caption{A force--free strip and the corresponding tangent curve.}
 \end{center}
 \end{figure}
 
In fact, there is a more elegant way to describe force--free strips without making use of the calculus of variation and differential equations we required for the previous arguments.\\
 We will look at the tangent image $T$ of a Frenet curve $\gamma$ in $\mathbb{R}^3$ as a regular curve, i.e. as an equivalence class of parameterizations. Then there are many (not necessarily arclength parametrized) curves $\tilde{\gamma}$$\colon$$[0,\tilde{L}]\mapsto \mathbb{R}^3$ with the same tangent image $T$. We temporarily fix the tangent image and minimize $S_{-1}$ among all curves with the same tangent image.\\

\begin{3.theorem}
 Let $T$$\colon$$[0,\tilde{L}]\mapsto S^2$ be an arclength parametrized spherical curve with curvature
 $\lambda$. Then among all Frenet curves with tangent image $T$ the curve (\ref{gamma}) minimizes $S_{-1}$
 and is unique up to translations.
\end{3.theorem}
 \begin{proof}
 For any curve $\tilde{\gamma}$ with tangent image $T$ and curvature $\tilde{\kappa}$ we have by the inequality between arithmetic and geometric mean 
 \begin{align*}\begin{split}
 S_{-1}(\tilde{\gamma})&=\int_{0}^{\tilde{L}}(\tilde{\kappa}^2(1+\lambda^2 )^2+1)\frac{1}{\tilde{\kappa}} dt \\
               &=\int_{0}^{\tilde{L}}(\tilde{\kappa}(1+\lambda^2 )^2+\frac{1}{\tilde{\kappa}}) dt\\
               &\geq\int_{0}^{\tilde{L}}2(1+\lambda^2)dt\\
               &=2L(\gamma)\\
               &=S_{-1}(\gamma).
               \end{split}\end{align*}
 \end{proof}
 For force--free strips the bending energy is critical even if the end points of $\gamma$ are allowed to move, since the force vector $b_0$ comes from the boundary terms of the first variational formula. This implies that the boundary term drops out automatically. There are no conditions on the end points of $\gamma$ and the variational
 problem of $\gamma$ can be reduced to a variational problem on the tangent image. Consequently one can deduce Theorem 2 from the previous theorem, since $\int_{0}^{\tilde{L}}2(1+\lambda^2)dt$ has a critical value iff $T$ is a spherical elastic curve with Lagrange multiplier 1. Langer and Singer [3] showed that there are infinitely many closed curves minimizing $\int_{0}^{\tilde{L}}2(1+\lambda^2)dt$. Each such curve defines a closed force--free strip, since the mass center of $T$ is zero.
 In [\ref{ulrich}], it is shown that each closed spherical elastic curve with Lagrange multiplier 1 corresponds to a Willmore torus in $S^3$. More precisely, let $T$ be such a spherical curve, then we can parametrize all possible adapted frames (lifted to $S^3$) along the curve (\ref{gamma}) by the frame cylinder $F:[0,\tilde{L}]\times S^1\rightarrow S^3$. $F$ is the preimage of the tangent image $T$ under the Hopf map $S^3\rightarrow S^2$ described in [\ref{ulrich}]. We obtain the following
 \begin{corollary}
 The frame cylinder of a force--free strip is Willmore in $S^3.$
 \end{corollary}
We generalize the previous method and reduce the variational problem to spherical curves.
\begin{definition}

For an arclength parametrized spherical curve $T$ we call
 \begin{align}
 P(T):=2\int_{0}^{\tilde{L}}\sqrt{\left\langle T,b_0\right\rangle-\mu}(1+\lambda^2)ds
 \end{align}
 the P-functional.
\end{definition}
 \begin{3.theorem}
For a critical point $\gamma$ of the modified Sadowsky functional the tangent vector $T$ with spherical curvature $\lambda$ is a critical point of the P-functional.
 \end{3.theorem}
 \begin{proof}
 Let $T$$\colon$$[0,\tilde{L}]\mapsto S^2$ be an arclength parametrized spherical curve with curvature
 $\lambda$. For any function $\kappa$$\colon$$[0,\tilde{L}]\mapsto \mathbb{R}^+$, we can define a regular space curve 
 \begin{align}
 \gamma(t)=\int_{0}^{{t}}\frac{1}{\kappa}T ds.
 \end{align}
 $\gamma$ has curvature $\kappa$ and the Sadowsky functional of $\gamma$ is given by 
 \begin{align}\label{sa}
 S(\kappa)=\int_{0}^{\tilde{L}}\kappa (1+\lambda^2)^2 ds .
 \end{align}
 We want to look for critical points of $S$ when $T$ is held fixed (only $\kappa$ varies). We do these variations of $\gamma$
 under two constraints: The length
 \begin{align}
 L=\int_{0}^{\tilde{L}}\frac{1}{\kappa}ds
 \end{align}
 and the end points of $\gamma$
 \begin{align}
 \gamma(\tilde{L})=\int_{0}^{\tilde{L}}\frac{1}{\kappa} T ds
 \end{align}
 will be held fixed. These four scalar constraints allow us to add four Lagrange multipliers to the functional (\ref{sa}), conventionally 
 gathered into a scalar $\mu$ and a vector $b_0\in\mathbb{R}^3$:
 \begin{align}
 P_T(\kappa)=\int_{0}^{\tilde{L}}(\kappa(1+\lambda^2)^2-\frac{\mu}{\kappa}+\frac{\left\langle T,b_0\right\rangle}{\kappa})ds.
 \end{align}
 Varying $\kappa$ yields 
 \begin{align}
 \dot{P}_T=\int_{0}^{\tilde{L}}(\dot{\kappa}((1+\lambda^2)^2+\frac{\mu}{\kappa^2}-\frac{\left\langle T,b_0\right\rangle}{\kappa^2}))ds,
 \end{align}
 so $\kappa$ is critical for $P_T$ iff 
 \begin{align}\label{ka}
 \left\langle T,b_0\right\rangle=\kappa^2(1+\lambda^2)^2+\mu.
 \end{align}
 Computing $\kappa$ from (\ref{ka}) yields
 \begin{align}\label{super}
 \kappa=\frac{\sqrt{\left\langle T,b_0\right\rangle-\mu}}{1+\lambda^2}.
 \end{align}
  Then $P_T$ becomes 
  \begin{align}
 P(T)=2\int_{0}^{\tilde{L}}\sqrt{\left\langle T,b_0\right\rangle-\mu}(1+\lambda^2)ds.
 \end{align}
 This shows that for critical point $\gamma$ of $S_\mu $ 
 the tangent image $T$ is a critical point of the P-functional.
 \end{proof}

\end{document}